\newcounter{version}
\newcounter{arxiv}
\newcounter{bells}
\newcounter{full}
\newcounter{minimal}
\stepcounter{version}
\makeatletter{}%
\makeatletter{}%
\documentclass[10pt]{amsart}

\usepackage{amssymb}
\usepackage{mdwlist}
\usepackage{mathpazo}

\usepackage[final]{microtype}
\usepackage{bm}
\usepackage{ifthen}
\usepackage{calc}

\usepackage[utf8]{inputenc}
\usepackage[vietnamese, english]{babel}
\usepackage{textcomp}
\usepackage{xcolor}
\usepackage[pdftex]{graphicx}
\usepackage{xspace}
\usepackage{tikz}
\usepgflibrary{shapes}
\usetikzlibrary{arrows,shapes,backgrounds,decorations,snakes,calc,math,matrix}

\usepackage{array}
\usepackage{multirow}

\long\def\commentout#1{}

\newif\ifprint
\printfalse
\ifprint
	\definecolor{linkred}{rgb}{0,0,0} %
	\definecolor{linkblue}{rgb}{0,0,0} %
\else
	\definecolor{linkred}{rgb}{0.7,0.2,0.2}
	\definecolor{linkblue}{rgb}{0,0.2,0.6}
\fi
\definecolor{green}{rgb}{0,0.8,0.1}

\usepackage[enable]{backref}
\usepackage{xr-hyper}
\PassOptionsToPackage{obeyspaces}{url} 
\usepackage[
	hypertexnames=true,%
	hyperindex,
	pagebackref,
	pdftex,
	pdftitle={Sacks of dice with fair totals},
	pdfdisplaydoctitle,
	pdfpagemode=UseNone,
	breaklinks=true,
	extension=pdf,
	bookmarks=false,
	plainpages=false,
	colorlinks,
	linkcolor=linkblue,
	citecolor=linkred,
	urlcolor=linkred,
	pdfmenubar=true,
	pdftoolbar=true,
	pdfpagelabels,
	pdfpagelayout=SinglePage,
	pdfview=Fit,
	pdfstartview=Fit
]{hyperref}%

\usepackage{titlesec}
\titleformat{\section}{\normalfont\large\bfseries}{\thesection.}{0.5em}{}[\kern0.em]
\titleformat{\subsection}[runin]{\normalfont\bfseries}{\thesubsection.}{0.2em}{}[\kern0.5em]
\titleformat{\subsubsection}[runin]{\normalfont\bfseries}{\thesubsubsection.}{0.1em}{}[\kern0.5em]

\makeatletter

\def\maketitle{\par
  \@topnum\z@ %
  \@setcopyright
  \thispagestyle{firstpage}%
  \ifx\@empty\shortauthors \let\shortauthors\shorttitle
  \else \andify\shortauthors
  \fi
  \@maketitle@hook
  \begingroup
  \@maketitle
  \toks@\@xp{\shortauthors}\@temptokena\@xp{\shorttitle}%
  \toks4{\def\\{ \ignorespaces}}%
  \edef\@tempa{%
    \@nx\markboth{\the\toks4
      \@nx{\the\toks@}}{\the\@temptokena}}%
  \@tempa
  \endgroup
  \c@footnote\z@
  \@cleartopmattertags
}
\def\@settitle{\begin{center}%
  \baselineskip14\p@\relax
    \bfseries
\Large%
\@title%
  \end{center}%
}
\def\author@andify{%
  \nxandlist {\unskip ,\penalty-1 \space\ignorespaces}%
    {\unskip {} \@@and~}%
    {\unskip ,\penalty-2 \space \@@and~}%
}
\def\@setauthors{%
  \begingroup
  \def\thanks{\protect\thanks@warning}%
  \trivlist
  \centering\footnotesize \@topsep30\p@\relax
  \advance\@topsep by -\baselineskip
  \item\relax
  \author@andify\authors
  \def\\{\protect\linebreak}%
{\large\authors}%
  \ifx\@empty\contribs
  \else
    ,\penalty-3 \space \@setcontribs
    \@closetoccontribs
  \fi
  \endtrivlist
  \endgroup
}
\def\@setaddresses{\par
  \nobreak \begingroup
\footnotesize
  \def\author##1{\nobreak\addvspace\bigskipamount}%
  \def\\{\unskip, \ignorespaces}%
  \interlinepenalty\@M
  \def\address##1##2{\begingroup
    \par\addvspace\bigskipamount\indent
    \@ifnotempty{##1}{(\ignorespaces##1\unskip) }%
    {\ignorespaces##2}\par\endgroup}%
  \def\curraddr##1##2{\begingroup
    \@ifnotempty{##2}{\nobreak\indent\curraddrname
      \@ifnotempty{##1}{, \ignorespaces##1\unskip}\/:\space
      ##2\par}\endgroup}%
  \def\email##1##2{\begingroup
    \@ifnotempty{##2}{\nobreak\indent\emailaddrname
      \@ifnotempty{##1}{, \ignorespaces##1\unskip}\/:\space
      \ttfamily##2\par}\endgroup}%
  \def\urladdr##1##2{\begingroup
    \def~{\char`\~}%
    \@ifnotempty{##2}{\nobreak\indent\urladdrname
      \@ifnotempty{##1}{, \ignorespaces##1\unskip}\/:\space
      \ttfamily##2\par}\endgroup}%
  \addresses
  \endgroup
}
\renewenvironment{abstract}{%
  \ifx\maketitle\relax
    \ClassWarning{\@classname}{Abstract should precede
      \protect\maketitle\space in AMS document classes; reported}%
  \fi
  \global\setbox\abstractbox=\vtop \bgroup
    \normalfont\Small
    \list{}{\labelwidth\z@
      \leftmargin3pc \rightmargin\leftmargin
      \listparindent\normalparindent \itemindent\z@
      \parsep\z@ \@plus\p@
      
    }%
    \item[\hskip\labelsep\bfseries\abstractname.]%
}{%
  \endlist\egroup
  \ifx\@setabstract\relax \@setabstracta \fi
}

\def\section{\@startsection{section}{1}%
  \z@{.7\linespacing\@plus\linespacing}{.5\linespacing}%
  {\normalfont\bfseries\centering}}%

\renewcommand*\l@section{\@tocline{1}{0pt}{0em}{1.75em}{}}
\renewcommand*\l@subsection{\@tocline{2}{0pt}{1.75em}{2em}{}}

\def\ps@normal{\def\@oddhead{\small \hfill \leftmark \hfill\thepage }
\def\@evenhead{\small \thepage \hfill \rightmark \hfill}
\def\@oddfoot{}
\def\@evenfoot{}}
\commentout{

}

\AtBeginDocument{%
\setlength\abovedisplayskip{4pt}
\setlength\belowdisplayskip{4pt}
\setlength\abovedisplayshortskip{0pt}
\setlength\belowdisplayshortskip{0pt}}
\makeatother

\usepackage[backrefs,msc-links,nobysame]{amsrefs}

\makeatletter
\def\bib@div@mark#1{%
 \@mkboth{{#1}}{{#1}}%
	}
\def\print@backrefs#1{%
 \space\SentenceSpace$\leftarrow$\csname br@#1\endcsname
}

\commentout{%
\catcode`\'=11 %
\renewcommand{\PrintAuthors}[1]{%
 \ifx\previous@primary\current@primary
  \sameauthors\@empty
 \else
  \def\current@bibfield{\bib'author}%
		  \PrintNames{}{}{\scshape #1}%
 \fi
}
\catcode`\'=12 %
}%
\def\MRhref#1#2{%
 \begingroup
 \parse@MR#1 ()\@empty\@nil%
  \href{\MR@url}{\texttt{\@tempd\vphantom{()}}}%
  \ifx\@tempe\@empty
  \else
   \ \href{\MR@url}{\texttt{(\@tempe)}}%
  \fi
 \endgroup
}%
\def\MR#1{%
 \relax\ifhmode\unskip\spacefactor3000 \space\fi
 \begingroup
 \strip@MRprefix#1\@nil
  \edef\@tempa{\@nx\MRhref{MR\@tempa}{\@tempa}}%
 \@xp\endgroup
 \@tempa
}
\makeatother

\newcommand{\loosespread}{\linespread{1.12}\normalfont\selectfont}

\newcommand{\bibspread}{\linespread{1.05}\normalfont\selectfont}

\newcommand{\noparskip}{\parskip0pt}

\makeatletter
\newtheoremstyle{questions}{\thm@preskip}{\thm@preskip}%
     {\itshape}%
     {}%
     {\bfseries}%
     {.}%
     {6pt}%
     {\thmname{#1}\thmnumber{ #2}\thmnote{ \textbf{~[#3]}}}%
\makeatother

\theoremstyle{plain}
\newtheorem{thm}[equation]{Theorem}
\newtheorem*{GKthm}{Gasarch--Kruskal Theorem}
\newtheorem{prop}[equation]{Proposition}
\newtheorem{lem}[equation]{Lemma}
\newtheorem{cor}[equation]{Corollary}

\newtheorem{clm}[equation]{Claim}

\theoremstyle{questions}%

\theoremstyle{definition}
\newtheorem{definition}[equation]{Definition}

\newtheorem{rem}[equation]{Remark}%

\theoremstyle{remark}

\newcommand{\numberedinsertiontop}[4]{%
\vskip3pt%
\begin{center}\refstepcounter{equation}\label{#1}%
\centerline{\textbf{#2~\ref{#1}}~ #3}%
{#4}%
\end{center}%
\vskip3pt%
}%

\newcommand{\numberedinsertionbot}[4]{%
\vskip6pt%
\begin{center}\refstepcounter{equation}\label{#1}%
{#4}\par%
\centerline{\textbf{#2~\ref{#1}}~ #3}%
\end{center}%
\vskip3pt%
}%

\newcommand{\divides}{|}

\newcommand{\Aut}[1]{\operatorname{Aut}}
\newcommand{\RR}{\mathbb{R}}
\newcommand{\sq}[1]{[#1]}
\newcommand{\ang}[1]{\langle{#1}\rangle}
\newcommand{\TT}{\mathbf{T}}
\newcommand{\JJ}{\mathbf{J}}

\newcommand{\dd}{\mathbf{d}}

\newcommand{\jj}{\mathbf{j}}

\newcommand{\ee}{\mathbf{e}}
\renewcommand{\aa}{\mathbf{a}}
\renewcommand{\SS}{\mathbf{S}}

\newcommand{\sect}[1]{{section~\ref{#1}}}

\newcommand{\Sect}[1]{{Section~\ref{#1}}}

\numberwithin{equation}{section}

\def\calcbaseurl{http://faculty.fordham.edu/dswinarski/totaltoparts/}

\newif\ifauthornotes\authornotesfalse

\long\def\commentout#1{}%

\begin{document}
\loosespread
\noparskip 

\hbox{~}\vskip-36pt

\ifnum\value{version}>\value{arxiv}%
	\author[Ian Morrison]{{\normalsize\selectfont Ian Morrison}}
	\address{Department of Mathematics\\
	Fordham University\\
	Bronx, NY 10458}
	\email{morrison@fordham.edu}
\fi%

\title[Sacks of dice with fair totals]{Sacks of dice with fair totals}

\subjclass[2010]{Primary 60C05, 12D05}
\keywords{fair, dice}

\commentout{
\begin{abstract}
A fair sack is a finite set of independent dice, not required to be fair and allowed to have any number of sides, for which all totals are equally likely. These have been studied for over 60 years. Most results restrict the possible orders of dice in such a sack and almost no examples were known. Building on a rather different approach due to Gasarch and Kruskal, we give an explicit construction of all such sacks.
\end{abstract}}

\maketitle
\thispagestyle{empty}%

\section{Introduction}\label{introduction}

This paper gives a construction of all finite collections or sacks of independent dice such that, when the dice are rolled, all possible totals of the sides are equally likely. We begin with a brief history of work on the problem of characterizing such \emph{fair sacks}, then outline the plan of the paper.

Over 60 years ago, the familiar fact that a pair of fair cubical dice has totals that are unfair prompted J.~B.~Kelly to pose as \textsc{Monthly} problem E~925~\cite{Eninetwofive}, the converse question ``Can unfair dice have fair totals?'' which has a negative answer. A number of papers \cites{ChenRaoShreve, DudewiczDann, GasarchKruskal, Eninetwofive, MorrisonSwinarski, Parzen}, reviewed at the end of \sect{background},
have considered this question for more general sacks. Most give conditions on the orders of the dice in a sack that guarantee unfairness and very few examples of fair sacks were known. 

Gasarch and Kruskal~\cite{GasarchKruskal}, however, asked, ``Do all fair sacks share some common structure?'' They found local and global answers that are explained in detail in  Section 2. Locally, all dice in a fair sack must themselves be \emph{semifair}.\footnote{Gasarch and Kruskal use the less suggestive term ``nice.''} Globally, the sack must satisfy \emph{Uniqueness of Totals}: exactly one roll yields each total. Their work does not provide any way to test a sack of semifair dice for this global property other than brute force enumeration of the totals of all rolls\footnote{In \sect{applications}, we give an improved algorithm to test sacks of semifair dice for fairness.} and, although they gave examples of fair sacks with this property, they found no systematic construction. 

Our main result is a canonical construction of every fair sack. Here is a precis of how we proceed. \Sect{construction} gives a fuller guide, illustrating the steps by rostering all fair sacks with largest total $t=12$---the smallest $t$ that reveals all the wrinkles of the general case---and explaining, without proof, how their construction generalizes. Details and proofs of the general constructions are given in \sect{details}. We start from the observation that ``fair sacks give factorizations of $t$.'' Informally, we would like to invert this association but it is many-to-one, so we proceed in stages, enhancing factorizations by first ordering the factors, and then adding an auxiliary partition. Corollary~\ref{sacksection} produces an injective map from ordered factorizations to fair sacks and Proposition~\ref{mainconstruction} extends this map to all partitioned factorizations. The extension is no longer injective, but Theorem~\ref{intervalfreeunique} shows that restricting to interval free partitions (Definition~\ref{intervalfree}) gives an injective map with the same image. In \Sect{maintheorems}, after a motivating example again with $t=12$, we prove our main result, Theorem~\ref{mainthm}, which shows that this restriction is also a surjection---that is, Theorem~\ref{intervalfreeunique} constructs \emph{all} fair sacks. In \sect{applications}, we give a few applications.
Our methods are completely elementary, relying principally on a systematic exploitation of Uniqueness of Totals.

\section{The Gasarch--Kruskal theorem.}\label{background}

In this section, we define notions and notation used in the sequel, then state and prove the Gasarch--Kruskal theorem streamlining the original arguments slightly.

A \emph{die} $\dd$ of order $n \ge 2$ is a finite probability space whose sample space is the set $\ang{n}:= \{0, 1, \ldots , n-1 \}$ but that may have \emph{any} probability distribution. Indexing by $\ang{n}$ rather than the standard $\sq{n}:= \{1, 2, \ldots, n\}$ simplifies many formulae in the sequel. We use the terms \emph{roll} and \emph{side} as synonyms for \emph{trial} and \emph{outcome}, respectively, motivated by the example of standard cubical dice. However, our dice often have sides with probability $0$, so a better mental model is a spinner mounted over a circle divided into $n$ arbitrary sectors. The language of dice is historical in our problem. 

We index sides of dice by $j$ and denote the probability of  side $j$ by ${p}_\dd(j)$, omitting the $\dd$ when possible. We will confound the die $\dd$, the tautological random variable whose value on side $j$ is $j$, and the \emph{die polynomial} $\dd(x) := \sum_{j=0}^{n-1} {p}(j)x^j$ which is the generating function of this random variable. For example, a standard fair cubical die has $\dd(x) := \frac{1}{6}(1+x+x^2+x^3+x^4+x^5)$. As in this example, we will always write such polynomials with degrees increasing from left to right.

\begin{definition}\label{semifair} A die is \emph{semifair} if:
	\begin{enumerate} 
		\item\label{semifair1} Each ${p}(j)$ is either $0$ or equal to ${p}(0)$, which must thus be nonzero.
		\item\label{semifair2} It is \emph{palindromic}: that is, ${p}(n-j-1)={p}(j)$.
	\end{enumerate}
\end{definition}

\begin{rem}\label{semifairem} A few remarks about semifairness are in order. 
	\begin{enumerate}
		\item \label{sfa} Henceforth, we abuse  notation by rescaling semifair dice so that $p(0)$, the common value of the nonzero $p(j)$, is $1$. Since the probability condition that the unscaled $p(j)$ sum to exactly $1$ allows us to reverse the scaling, we lose nothing by assuming this. Doing so allows us to avoid denominators and be able to work with monic polynomials throughout.
		\item \label{sfb} Set $\Psi_t(x) := 1+x+\dots+x^{t-1} = \bigl(\frac{1-x^{t}}{1-x}\bigr)$. The first form shows that  $\Psi_t(x)$ is the polynomial of a fair die of order $t$---see the cubical example above---and the second that its roots are exactly the $t${th} roots of unity, except for~$1$.
		\item A die $\dd$ of order $n$ is semifair if and only if $\dd(x)$ is obtained from $\Psi_n(x)$ by setting to $0$ a palindromic set of the interior coefficients.
	\end{enumerate}
\end{rem}

A \emph{sack} $\SS$ of size $m_\SS$ is a set of \emph{independent} dice $\dd_i$ of orders $n_i \ge 2$ indexed by $i \in \sq{m_\SS}$. To simplify notation, we omit reference to $\SS$ when it is understood and write, for example, $m$ for $m_\SS$. Such an $\SS$ has a product sample space $\JJ$ indexed  by rolls $\jj = (j_1, j_2, \ldots, j_m) \in \prod_{i\in \sq{m}} \ang{n_i}$ that carries, by independence, the product probability distribution 
$p(\jj) = \prod_{i \in \sq{m}} p_{\dd_i}(j_i)$.

On $\JJ$, we have independent random variables for each die $\dd_i$  whose value on any roll $\jj$ is $j_i$ and whose generating function is thus the die polynomial $\dd_i(x)$. We sum these to get the \emph{total random variable} $\TT(\jj) := \sum_{i \in \sq{m}}j_i$ which takes on the $t$ values in $\ang{t}$, where $t -1 := \sum_{i \in \sq{m}} (n_i-1) $. 

Since the generating function of a sum of independent random variables is the product of the generating functions of its terms (see \cite[p.~180, Theorem~6]{Chung}), the total $\TT$ has generating function
\begin{equation}\label{prodequation}
\TT(x) = \prod_{i \in \sq{m}} \dd_i(x) = \sum_{s = 0}^{t-1}\biggl( \sum_{T(\jj) = s} p(\jj)\biggr)x^s\,. 
\end{equation}
For two standard dice, this is a shifted form of the  familiar formula for totals:
{\small\[1{\kern0.5pt+\kern0.5pt}2x{\kern0.5pt+\kern0.5pt}3x^2{\kern0.5pt+\kern0.5pt}4x^3{\kern0.5pt+\kern0.5pt}5x^4{\kern0.5pt+\kern0.5pt}6x^5{\kern0.5pt+\kern0.5pt}5x^6{\kern0.5pt+\kern0.5pt}4x^7{\kern0.5pt+\kern0.5pt}3x^8{\kern0.5pt+\kern0.5pt}2x^9{\kern0.5pt+\kern0.5pt}x^{10}=(1{\kern0.5pt+\kern0.5pt}x{\kern0.5pt+\kern0.5pt}x^2{\kern0.5pt+\kern0.5pt}x^3{\kern0.5pt+\kern0.5pt}x^4{\kern0.5pt+\kern0.5pt}x^5)^2.\]}%
A fair sack is simply one for which $\TT(x)=\Psi_t(x)$ (see Remark~\ref{semifairem}.\ref{sfb}).

\begin{GKthm}[{{\cite[Corollary~5]{GasarchKruskal}}}] A sack is fair if and only if:		
	\begin{enumerate}
		\item\label{sfd} Each die in it is \emph{semifair}.
		\item\label{gku} \textbf{\textup{(Uniqueness of Totals)}}~ Each total is obtained from a unique \emph{effective} roll.	
	\end{enumerate}
\end{GKthm}

We first show that the dice $\dd(x)$ in a fair sack must be semifair which is the heart of the theorem. Palandromicity is easy. If the sack has order $t$, let $\zeta $ be a primitive $t\,${th} root of unity. The irreducible real factors of $\Psi_t(x)$ are $x+1$, when $t$ is even, and $x^2-(\zeta^j+\zeta^{-j})x+1$, for $j
\in \sq{\lfloor \frac{t-1}{2}\rfloor}$. These are palindromic and $\dd(t)$, being real, must be a product of them, so it is palindromic. The key step is the following.
\begin{lem}[{{\cite[Lemma~3]{GasarchKruskal}}}]\label{semipalinduction}
If $\dd(x) := \dd'(x)\cdot \dd''(x)$ is semifair and both $\dd'$ and $\dd''$ are palindromic, then both $\dd'$ and $\dd''$ are semifair.
\end{lem}

Given the lemma, an induction on the size $s$ of a fair sack $\SS$ shows semifairness of its dice. The case $s=1$ is trivial. If $s\ge 2$, just take any two dice $\dd'(x)$ and $\dd''(x)$  and replace them by their product $\dd(x)$, getting a fair sack of smaller size whose dice, in particular $\dd(x)$, must inductively  be semifair. Since we know already that $\dd'$ and $\dd''$ are palindromic, the lemma then shows that both are also semifair.

\begin{proof}[Proof of Lemma~\ref{semipalinduction}]
Without loss of generality, assume that $n' \le n''$. 
\begin{clm}\label{zeroclaim}~ 
For $ j \in \sq{n'-1}$, either $p'(j)=0$ or $p''(j)=0$.
\end{clm}
Since all the coefficients are nonnegative, the claim  will follow if we show that $\sum_{j=1}^{n'-1}p'(j) p''({j}) \le 0$. To see this, use the palandromicity of $\dd'$ once to write
\[ 
\sum_{j=1}^{n'-1}p'(j) p''({j}) = \sum_{j=1}^{n'-1}p'({n'-1-j}) p''(j) = p({n'-1}) -p'({n'-1}) p''(0)\,,
\]
and then a second time to write
\[ 
 p({n'-1}) -p'({n'-1}) p''(0) = p({n'-1}) -p'(0)p''(0) = p({n'-1}) -p(0) \le 0\,,
\]
with the last inequality following because $\dd(x)$ is semifair. 

For notational convenience, we define $p'(j)=0$ for $n'\le j <n''$. With this convention, Claim~\ref{zeroclaim} then holds for $ j \in \sq{n''-1}$. By palindromicity and monicity $p'(0)= p''(0) =1$, so we can expand
\begin{equation}\label{pjexp}
p(j) :=  \sum_{i=0}^{j}p'(i) p''({j-i})= p'(j) +  \sum_{i=1}^{j-1}p'(i) p''({j-i})  +p''(j)\,.
\end{equation}

We will use this expansion to show, by induction on $j \in \sq{n''}$, that each of $p'(j)$ and $p''(j)$ is either $0$ or $1$. 
By hypothesis, $p(j)$ and, by induction, all the terms in the middle sum in~\eqref{pjexp} are either $0$ or $1$. Since all terms are nonnegative, if $p(j)=0$, then all the terms in the sum as well as both $p'(j)$ and $p''(j)$ must also be $0$. If $p(j)=1$, there are two possibilities. Either exactly one term in the sum is $1$ and both $p'(j)$ and $p''(j)$ are $0$, or, all the terms in the sum are $0$ and $p'(j)+p''(j)=1$. But one of $p'(j)$ and $p''(j)$ is $0$ by the claim, so the other must then equal~$1$. 
\end{proof}
	
Semifairness of its dice is necessary but far from sufficient for the fairness of a sack. For example, although standard dice are semifair, a pair is an unfair sack.	Indeed, Corollary~\ref{degreedivisibility} shows that most semifair dice do not lie in \emph{any} fair sack.

To see that the additional global property \emph{Uniqueness of Totals} is both necessary and sufficient for fairness, we ask what~\eqref{pjexp} implies about a sack $\SS$ of semifair dice. Since all $p_i(j)$ are either $0$ or $1$, each product $p(\jj)$ is also either $0$ or $1$. In the latter case, we call $\jj$ \emph{effective} and must have $p_i(j_i)=1$ for all $i$. Thus, the coefficient of $x^s$ in $\TT(x)$ simply counts the number of effective rolls with total $s$. So the sack $\SS$ is \emph{fair} and all totals are equally likely if and only if each total arises from the \emph{same} number of effective rolls. Since the total $0$ arises from exactly one effective roll, with all $j_i=0$, \emph{every} total must arise from exactly one effective roll---Uniqueness of Totals---and $\TT(x) = \Psi_t(x)$.

Our other basic tool here is an easy consequence, not mentioned in \cite{GasarchKruskal}. If a fair sack contains a die with a nonzero $x^s$ term for any $s>0$, then the total $s$ arises from rolling $s$ on this die and $0$ on all the others. By Uniqueness of Totals, we deduce the following.

\begin{cor}\label{uniqueterms} \textbf{\textup{(Uniqueness of Terms)}}~  A fair sack can contain at most one die with nonzero $x^s$ term for any $s \in \sq{t}$ and contains such a die if and only if $x^s$ does \emph{not} arise as a product of terms of strictly lower degree. In particular, there is always a unique die with nonzero $x$ term.
\end{cor}

Since a semifair die of order $n$ has a nonzero $x^{n-1}$ term, we get the following.

\begin{cor}\label{orderrestrictions}  No two dice in a fair sack can have the \emph{same} order. \end{cor}
	We digress for a moment to document work of several earlier authors (most mutually unaware of each other) on special cases of this result. Almost all the arguments use inequalities involving the side probabilities to reach a contradiction. This is the approach of Moser and Wahab~\cite{Eninetwofive} to show there is no fair pair of dice of order $6$. Dudewicz and Dann~\cite{DudewiczDann},\footnote{Their title suggests, incorrectly, that no fair sacks exist, and they make a mysterious claim in the last line of the paper that ``similar results'' hold for general sacks.} although they do not cite~\cite{Eninetwofive}, note that, for  identical cubical dice, the conclusion is a ``well-known'' exercise and cite the text of Parzen~\cite{Parzen}, where this is Problem 9.12.\footnote{That no fair sack consisting of two dice of order $n$ exists is also, as noted by the referee, ``well known.'' The reader may enjoy checking this. Hint: Obtain a contradiction by showing that the total $n-1$ has probability at least $p_0p'_{n-1}+p_{n-1}p'_0=\frac{1}{2n-1}\bigl(\frac{p_0}{p_{n-1}} +\frac{p_{n-1}}{p_0}\bigr) \ge \frac{2}{2n-1}$, seeing the equality by observing that fairness implies $p_{n-1}p'_{n-1}=p_0p'_{0}=\frac{1}{2n-1}$ and the inequality by using a bit of calculus.} They prove that no fair sack (other than a singleton) can have \emph{all} dice of equal order $n$ by showing that the total $n-1$ must have probability strictly greater than $\frac{1}{t}$. Their result is reproved (but not cited) by Chen, Rao and Shreve~\cite{ChenRaoShreve} by showing that there must be a pair of totals whose probabilities differ by at least $\bigl\lvert\frac{m-1}{m^2 n}\bigr\rvert$. The stronger claim that all orders must be distinct was first proved by Gasarch and Kruskal~\cite{GasarchKruskal} (although they cite, incorrectly, \cite{ChenRaoShreve}), by casting the argument for the $s=n-1$ case of Corollary~\ref{uniqueterms} as a series of inequalities.

We should also mention an overlapping result. No fair sack can contain more than one die of even order. Such dice have polynomials of odd degree which must have a real root. But $\Psi_t(x)$ has no real roots for odd $t$ and exactly one for even $t$. This argument first occurs in the proof of Finch and Halmos~\cite{Eninetwofive} that there is no fair pair of dice of order $6$ and is also found in~\cite{GasarchKruskal} and~\cite{MorrisonSwinarski} for other even orders.

\section{Guide to the constructions and roster of fair sacks with total $\mathbf{12}$.}\label{construction} 

\subsection*{From fair sacks to unordered factorizations.}

A \emph{factorization} of $t$ of length $\ell$ will be a tuple $\aa :=(a_1, a_2, \ldots, a_\ell)$, usually indexed by $h$ and viewed as \emph{ordered}, for which
\begin{equation}\label{factortplus1}
	\prod_{h \in \sq{\ell}} a_h = t, \text{~~~and with each $a_h$ at least $2$.}
\end{equation} 
Note that we do \emph{not} require the $a_h$ to be prime. 

We start by noting that the Gasarch--Kruskal Theorem implies that any fair sack $\SS$ yields an \emph{unordered} factorization $\aa$ of length equal to the order of $\SS$ by taking $a_h$ to be the number of nonzero coefficients of $\dd_h(x)$. In (\ref{factortplus1}), the equation holds because each side counts the number of nonzero terms $p(\jj)$ in (\ref{prodequation}) and the inequalities on the $a_h$ hold by Remark~\ref{semifairem}.\ref{sfa}. We immediately get the last statement of~\cite[Corollary~9]{GasarchKruskal}: for $t$ prime, the only fair sack is a single fair $t$-die. Simply put, ``fair sacks give unordered factorizations.''

\subsection*{From ordered factorizations to fair sacks.}

The next step is to show that ``\emph{ordered} factorizations give fair sacks.'' More precisely, Corollary~\ref{sacksection} constructs, from each ordered factorization $\aa$, a factorization sack $\SS_\aa$ of size $\ell$ with dice $\dd_h(x):=\Psi_{a_h}(x^{b_h})$, where $b_h := \prod_{h' < h} a_{h'}$. Table~\ref{twelvefair} shows the sacks that arise for $t=12$.

\enlargethispage{\baselineskip}
\numberedinsertiontop{twelvefair}{Table}{Ordered factorizations $\aa$ of $12$ and their fair sacks $\SS_\aa$}{\nopagebreak%
	\begin{tabular}{>{$}c<{$}>{$}c<{$}}
		 a_1\cdot a_2 \cdots a_\ell  & \dd_{1}(x)\cdot \dd_{2}(x) \cdots \dd_{\ell}(x))\\
	2\cdot 2\cdot 3	& (1+x)(1+x^2)(1+x^4+x^8) \\
	2\cdot 3\cdot 2 & (1+x)(1+x^2+x^4)(1+x^6) \\
	2\cdot 6 & (1+x)(1+x^2+x^4+x^6+x^8+x^{10}) \\
	3\cdot 2\cdot 2 & (1+x+x^2)(1+x^3)(1+x^6) \\
	3\cdot 4 & (1+x+x^2)(1+x^3+x^6+x^9) \\
	4\cdot 3 & (1+x+x^2+x^3)(1+x^4+x^8) \\
	6\cdot 2 & (1+x+x^2+x^3+x^4+x^5)(1+x^6) \\
	12 & \kern-8pt(1+x+x^2+x^3+x^4+x^5+x^6+x^7+x^8+x^9+x^{10}+x^{11})
	\end{tabular}
}

Lemma~\ref{partialfairness} and Corollary~\ref{sacksection} imply that the sacks $\SS_\aa$ are always fair. However, not all fair sacks arise from this construction.

\subsection*{Partition-factorization sacks.}

Further fair sacks can be produced from factorization sacks by a collapsing or subtotaling process in which we replace disjoint subsets of the dice by their total dice. Equivalently, by (\ref{prodequation}), we can replace the polynomials of the dice in the subset by their product. Such a collapsing is specified more precisely by a partition $\Pi := [\pi_i, \pi_2, \ldots, \pi_m]$ of $\sq{\ell}$ which we will view both as a disjoint union decomposition  $\{1,2,\ldots,\ell\} = \mathop{\dot\bigcup}_{g=1}^m \pi_g$ and as a surjective function from $\sq{\ell} \to \sq{m}$ with fiber $\pi_g$ over $g$. 

\begin{definition}\label{partfact} To each part $\pi_g$ of $\Pi$, associate a subtotal die $\dd_{g}(x) = \prod_{h \in \pi_g} \dd_h(x)$ which, by (\ref{prodequation}), is the total die of the subsack of $\SS_\aa$ associated to $\pi_g$ and, to the pair $(\aa,\Pi)$, the \emph{partition-factorization} sack $\SS_{\aa,\Pi}$ consisting of the subtotal dice $\dd_{g}(x)$ of the parts of $\Pi$. We say that $\SS_{\aa,\Pi}$ arises from $\aa$ via~$\Pi$.
\end{definition}

Proposition~\ref{mainconstruction} implies that all such sacks are fair and that all arise from factorizations with each $a_h$ prime. For $t=12$, we get three new fair sacks in this way, by using the partition $\Pi = [\{1, 3\},\{2\}]$ with the three length $3$ factorizations. From $2\cdot 2\cdot 3$, we get the factorization $(1+x+x^4+x^5+x^8+x^9)(1+x^2)$ while $2\cdot 3 \cdot 2$ and $3\cdot 2\cdot 2$ give $(1+x+x^6+x^7)(1+x^2+x^4)$  and  $(1+x+x^2+x^6+x^7+x^8)(1+x^3)$. This turns out to complete the roster of fair sacks with $t=12$.

There are other partition-factorization sacks, but they already appear in Table~\ref{twelvefair}. For example, for the ordered factorization $2\cdot 2\cdot 3$ whose corresponding dice factorization is $(1+x)(1+x^2)(1+x^4+x^8)$, we get the $4\cdot 3$ line in the Table from $\Pi = [\{1,2\},\{3\}]$. More generally, while the association $\aa\to\SS_\aa$ is injective, its extension $(\aa, \Pi)\to\SS_{\aa, \Pi}$ is not. Restricting to prime factorizations does not cure this: the $6\cdot 2$ line in Table~\ref{twelvefair} arises from both of the ordered factorizations $2\cdot 3\cdot 2$ and $3\cdot 2\cdot 2$ via the partition $\Pi= [\{1, 2\},\{3\}]$.

\subsection*{Interval-free partitions.}
Fortunately, there is a simple way, already suggested by the examples above, to obtain all  partition-factorization sacks in a unique way by restricting which partitions are used. 
\begin{definition}\label{intervalfree}
	A partition $\Pi$ of $\aa$ is \emph{interval free} if no part contains consecutive elements of $\sq{\ell}$.
\end{definition}

In the example with $t=12$ above, the interval free partitions are those in Table~\ref{twelvefair} with all parts singletons and the partition $\Pi = [\{1, 3\},\{2\}]$ that yielded new fair sacks from the three length $3$ factorizations. Although, a priori, requiring interval freeness eliminates only the ambiguity arising  from collapsing consecutive factors, Theorem~\ref{intervalfreeunique} shows that each partition-factorization sack arises uniquely from an interval free partition of a (possibly different) factorization. 

\subsection*{Why do we obtain all fair sacks?}
So far we have constructed lots of fair sacks using simple combinatorial observations and the reader will see in the next section that the proofs are all fairly straightforward. We see no a priori reason to expect that our constructions produce \emph{all} fair sacks. Our main result, Theorem~\ref{mainthm}, shows that, in fact, they do. We prove it in \sect{maintheorems}, by making a careful inspection of a general fair sack to extract from it the canonical factorization and interval free partition from which it arises. The details of the analysis are considerably more delicate than what comes before. 

\section{Details and proofs of the constructions.}\label{details}

In this section, we define general partition-factorization sacks and show they are fair. As for sacks, we try to simplify notation by omitting reference to the factorization and partition when possible. We begin with an easy but crucial lemma that shows that factorization sacks are fair.

\begin{lem}\label{partialfairness} Fix an ordered factorization $\aa :=(a_1, a_2, \ldots, a_\ell)$ of $t$ of length $\ell$. For $h \in \sq{\ell+1}$, define $b_h := \prod_{h' < h} a_{h'}$ and note that, by hypothesis, $t= b_{\ell+1}$. For $h \in \sq{\ell}$, define  $\dd_{h}(x) := \Psi_{a_h}(x^{b_h})$ and $\ee_{h}(x) := \prod_{h' \le h}\dd_{h'}(x)$. Then $\ee_{h}(x) = \Psi_{b_{h+1}}(x)$. In particular, $\ee_\ell(x) = \Psi_{t}(x)$.
\end{lem}
\begin{proof} Observe that the roots of $\dd_{h}(x)$ are exactly the $b_h${th} roots of all nontrivial $a_h${th} roots of unity or, equivalently, all $b_ha_h${th} roots of unity of order not dividing $b_h$ or, again equivalently, all the $b_{h+1}${st} roots of unity of order not dividing $b_h$. By induction on $h$, the roots of $\ee_h(x)$ are exactly the nontrivial $b_{h+1}${st} roots of unity. Since both sides are monic polynomials with the same roots, $\ee_h(x) = \Psi_{b_{h+1}}(x)$.
\end{proof}

\begin{cor}\label{sacksection} If $\aa$ is an ordered factorization of $t$ of length $\ell$, the \emph{factorization sack} $\SS_\aa$ of size $\ell$ whose dice are defined by $\dd_{h}(x) := \Psi_{a_h}(x^{b_h})$ is a fair sack with total $t$.
\end{cor}

We note an equation that follows from Lemma~\ref{partialfairness} by dividing the $h=v$ case by the $h=u$ case and canceling those $\dd_{h'}(x)$ that are factors of both $\ee_{v}(x)$ and $\ee_{u}(x)$.
\begin{equation}\label{intervaleq}
	\prod_{h = u}^{v}\dd_{h}(x) = \frac{\Psi_{b_{v+1}}(x)}{\Psi_{b_{u}}(x)}	
\end{equation}
This has a consequence that we will need later.

\begin{cor}\label{collapsecor} If $\aa$ is obtained from an ordered factorization $\aa'$  of length $\ell$ by replacing consecutive factors $a'_{u} \cdots a'_{v}$ by their product, then, $\dd_h(x) = \dd'_h(x)$ for $1\le h <u$,  $\dd_u(x) = \prod_{h=u}^{v}\dd'_h(x)$, and $\dd_h(x) = \dd'_{h+v-u}(x)$ for $u<h\le\ell-v+u$.
\end{cor}
\begin{proof}  By construction, we have $a_h = a'_h$ for $1\le h < u$, we have $a_{u} = \prod_{h=u}^{v}a'_h$, and we have $a_h = a_{h+v-u}$ for $u<h\le\ell-v+u$. Therefore $b_h = b'_h$ for $h \le u$, $b_{u+1} = b'_{u'} \prod_{h=u}^{v}a'_h = b'_{v+1}$ and $b_h = b_{h+v-u}$ for $u<h\le\ell-v+u$. Thus, only the formula for $\dd_u(x)$ is not immediate. We may view $\dd_u(x)$ as the left side of (\ref{intervaleq}) applied to $\aa$ with $v= u$ and the product $\prod_{h=u}^{v}\dd'_h(x)$ as the left side of (\ref{intervaleq}) applied to $\aa'$. The formula just given for $b_{u+1}$ says that these two instances of (\ref{intervaleq}) have equal right-hand sides. Hence they have equal left-hand sides. 
\end{proof}

Next, we check that partition-factorization sacks are fair and that all arise, though generally in many ways, from partitions of prime factorizations.

\begin{prop}\label{mainconstruction}$\,$ 
	\begin{enumerate}
		\item Every partition-factorization sack is fair. 
		\item \label{fpprime} Every partition-factorization sack arises from an ordered \emph{prime} factorization.
		\end{enumerate} 
\end{prop}
\begin{proof} Because $\Pi$ simply partitions the $\dd_h(x)$ into disjoint groups with products $\dd_{g}(x)$, the product of all the $\dd_h(x)$ and of all of the $\dd_{g}(x)$ are equal. The former, by Lemma~\ref{sacksection}, equals $\Psi_t(x)$. Hence, the sack $\SS_{\aa,\Pi}$ is also fair. 
	
	For (\ref{fpprime}), first construct an ordered prime factorization $\aa'$ by simply replacing each $a_h$ by an ordered prime factorization, writing the factors of $a_1$ first, then those of $a_2$ and so on. By an inductive application of Corollary~\ref{collapsecor}, the product of the dice polynomials associated to the prime factors of any $a_h$ equals $\dd_h(x)$. 
This implies that if $\Pi'$ is the partition with $m$ parts $\pi'_g$ each consisting of all the prime factors of the $a_h$ in $\pi_g$, then the dice associated to the $g${th} parts of $\Pi$ and $\Pi'$ are equal. 
\end{proof}

\begin{rem} Proposition~\ref{mainconstruction} can be used to construct all the fair sacks on p.137 of~\cite{GasarchKruskal}.\footnote{For the interested reader, here are the factorization (and the partition, if any parts are not singletons) giving each sack with its location on p.137 of \cite{GasarchKruskal} in parentheses: $2\cdot i$ (2); $i\cdot 2$ (3); $3\cdot i$ (4); $2\cdot 2\cdot i$ via $[\{1,3\},\{2\}]$ (5); $3\cdot 4$ and  $2\cdot2\cdot3$ via $[\{1,3\},\{2\}]$ (first paragraph after 5); $2\cdot 2 \ldots \cdot 2$ (second paragraph after 5).} 
\end{rem}

We close this section by checking that each partition-factorization sack is uniquely specified if we require interval freeness (see Definition~\ref{intervalfree}).

\begin{thm}\label{intervalfreeunique} Every partition-factorization sack $\SS$ arises from an interval free partition of an ordered factorization, both of which are uniquely determined by $\SS$.
\end{thm}
\begin{proof}
	Given a factorization and a partition of it, here is how to obtain from it a new factorization and an interval free partition without changing either the number of parts or any of the associated dice. If any of the given parts contains consecutive factors, replace these by their product in the factorization and assign this product factor to the part formerly containing the consecutive factors, leaving all other parts unchanged. The partition of the collapsed factorization that this produces is interval free. An application of Corollary~\ref{collapsecor} like that used in proving Proposition~\ref{mainconstruction}(\ref{fpprime}) shows that the dice associated to each of the corresponding old and new parts will be equal and, hence, they yield the same sack.

	We will prove the uniqueness of the interval free realization for a given sack $\SS$ by induction on the length $\ell$ of the factorization $\aa$. If this number is $1$, then we have a fair die. Otherwise, observe that, by Uniqueness of Terms (Corollary~\ref{uniqueterms}), there is a unique part $\pi_g$ whose die $d_{\pi_g}(x)$ has nonzero $x$ coefficient. In the construction of factorization sacks, only the die $d_1(x)$ has nonzero $x$-coefficient so $1$ must lie in $\pi_g$. We claim that $a_1$ is the smallest $s$ such that coefficient of $x^s$ in $d_{\pi_g}(x)$ equals $0$. No smaller power can have a zero coefficient because $d_1(x) = \Psi_{a_1}(x)$ is a factor of $d_{\pi_g}(x)$. Again, by construction, only the die $d_2(x)$ has nonzero $x^{a_1}$ coefficient. So if this coefficient were nonzero in $d_{\pi_g}(x)$, then $d_2(x)$ would be a factor and hence $2$ would also lie in $\pi_g$, contradicting the interval freeness of $\Pi$. Thus $\SS$ determines both $a_1$ and the index $g$ of the part containing $1$. 

	Now we replace $t$ by $t' :=\frac{t}{a_1}$, define an ordered factorization $\aa'$ of $t'$ by deleting $a_1$ from $\aa$, and define an interval free partition $\Pi'$ of $n-1$ by first deleting $1$ from $\pi_g$ (and deleting $\pi_g$ from $\Pi$ if it is now empty) and then shifting all parts left $1$. This yields an interval free realization of a sack $\SS'$, also determined by $\SS$, but with $\ell$ reduced by $1$. By induction, $\SS'$ determines $\aa'$ and $\Pi'$. But from $\aa'$ and $a_1$ we recover $\aa$. Similarly, from $\Pi'$ and the index $g$ of the part containing $1$ (or the fact that $1$ lay in a deleted part), we recover $\Pi$.
\end{proof}

\section{The main theorem.}\label{maintheorems}

The goal of this section is to prove that Theorem~\ref{intervalfreeunique} constructs all fair sacks.

\begin{thm}\label{mainthm} Every fair sack $\SS$ of size $m$ and total $t$ equals $\SS_{\aa,\Pi}$ for $\Pi$ a uniquely determined interval free partition with $m$ parts of an ordered factorization $\aa$ of $t$.
\end{thm}

To get a feel for how the argument goes and where the key difficulty lies, consider how we might reconstruct, given only the dice themselves, the factorization and interval free partition associated to the two sacks of total $12$ with $\aa= 2\cdot 3 \cdot 2$, one with $\Pi = [\{1, \}, \{2\},\{3\}]$ and dice $(1+x)\cdot (1+x^2+x^4) \cdot (1+x^6)$ and the other with  $\Pi = [\{1, 3\},\{2\}]$ and dice $(1+x+x^6+x^7)\cdot (1+x^2+x^4)$.

In both cases, Uniqueness of Terms (Corollary~\ref{uniqueterms}) locates the first die as the unique one with nonzero $x$ term and the first factor $a_1= 2$ as the smallest degree \emph{not} appearing in its polynomial. Likewise, the second die is the one which \emph{does} have an $x^2$ term and the second factor $a_2=3$ is the smallest integer such that this die has no term of degree $2a_2$. Note that, while both $a_1$ and $a_1\cdot a_2$ divide $t=12$, the way we chose them gives no guarantee that they must. In both cases, products of known terms exactly account for all totals $s$ less than $b_2 := a_1a_2 = 6$.  Again, Uniqueness of Terms thus guarantees that the only terms of degree less than $b_2=6$ in either die are those already known: $1+x$ in the first and $1+x^2+x^4$ in the second. To this point, the argument works in general, changing, of course, $a_1$ and $a_2$.
  
At this point, there must be a unique die with an $x^6$ term. It has no $x^{12}$ term, so we set $a_3=2$. Only now are we assured that the $a_h$ give a factorization of $12$. When $\Pi = [\{1, \}, \{2\},\{3\}]$, the $x^6$ term occurs in the third die and products of known terms account for all totals $s < 12$. Hence no other nonzero terms can occur, and we are done. But when $\Pi = [\{1, 3\},\{2\}]$, such products do not produce an $x^7$ term since now the known terms $x$ and $x^6$ both occur in the first die. So Uniqueness of Terms tells us that there must be an $x^7$ term in some die. The key point that must be checked is that this term must occur in the first die (and, more generally, in similar situations, in the \emph{same} die as the $x^6$ term we have just located). Once we know that it does, then products of known terms in the two dice we have constructed uniquely account for all totals $s < 12$,  and we are again done. 
 
Why can the $x^7$ term \emph{not} lie in the second die, nor in some potential third die that we have yet to encounter in reconstructing the sack? If it did, the total $x^8$ would arise in two ways, as the product of the $x^6$ term in the first die and the $x^2$ term in the second, and as the product of the $x$ term in the first die and the $x^7$ term in the second or third die. This cannot happen, again by Uniqueness of Terms. 

The proof of Theorem~\ref{mainthm} for general $\SS$ uses the same basic ideas. However, as the number of factors increases, we encounter interval free partitions with arbitrarily large parts, for which the number of ``missing'' terms like $x^7$ that must be shown to be correctly located grows exponentially. The need to set up an induction that both keeps track of all these ``missing'' terms and allows us to identify for each, a degree, like $8$ in the example above, for which Uniqueness of Terms would be violated if the ``missing'' term were incorrectly located motivates the following definition. 

\begin{definition}\label{partialdef} For an ordered $\ell$-tuple $\aa := (a_1, a_2, \ldots, a_{\ell})$ with each $a_h >1$, define, as above, $b_{h}= \prod_{h'<h}a_{h'}$ for $h \le \ell+1$ (with $b_1=1$). Such an $\aa$ together with a map $\Pi: \sq{\ell}\to\sq{m}$ (thought of as the set of dice in $\SS$) is a \emph{truncated realization} of $\SS$ if,
	\begin{enumerate}
		\item\label{p1}  For all $g \in \sq{m}$, 
			\begin{equation}\label{ematchesd}
		\dd_g(x) \equiv \prod_{h\in\Pi^{-1}(g)} ~\Psi_{a_h}(x^{b_h}) \mod\bigl(x^{b_{\ell+1}}\bigr)\,.
			\end{equation}
		\item\label{p2}  For $1\le h < \ell$, $\Pi(h)\not= \Pi(h+1)$.
		\item\label{p3} The $b_{\ell+1}$-term in $\dd_{\Pi(\ell)}(x)$ is zero. 
	\end{enumerate}
We say that $(\aa',\Pi')$ \emph{extends} $(\aa,\Pi)$ if the initial $\ell$ values of both $\aa'$ and $\Pi'$ match those of $\aa$ and $\Pi$. 
\end{definition}

\begin{rem}\label{partialrem}~ 
Intuitively, (\ref{p1}) says that the dice in $\SS_{\aa,\Pi}$ and the degree $b_{\ell+1}$ truncations of those in $\SS$ are matching fair sacks with total $b_{\ell+1}$, modulo dice in $\SS$ with trivial truncations; (\ref{p2}) says that $\Pi$ is interval free; and  (\ref{p3}) lets us choose an extension that preserves (\ref{p2}).
\end{rem}

Suppose that we have a truncated realization for which $b_{\ell+1}\ge t$. Then, by the preceding remark, we have $\SS = \SS_{\aa,\Pi}$, but now without truncation of $\SS$. Hence, we must have $b_{\ell+1}= t$ and, retrospectively, $\aa$ must be a factorization of $t$. Finally, since all dice in $\SS$ are nontrivial, $\Pi$ must be surjective and its fibers determine an interval free partition.
Thus, Theorem~\ref{mainthm} will follow by induction, with a trivial base case when $\ell=0$ once we prove the following.

\begin{clm}\label{mainclm} Any truncated realization $(\aa,\Pi)$ of $\SS$ of length $\ell-1$ with $b_{\ell}<t$ can be extended to a truncated realization $(\aa',\Pi')$ of length $\ell$.
\end{clm}

Given $(\aa,\Pi)$ of length $\ell-1$, we first find $\Pi'$. Remark~\ref{partialrem} implies that $x^{b_{\ell}}$ does not arise as a product of lower degree terms in the $\dd_g(x)$. By Uniqueness of Terms for $\SS$, there must be a unique die $\dd_{\gamma}(x)$ with nonzero $x^{b_{\ell}}$ term. We define $\Pi'(\ell)=\gamma$. The condition that the $b_{\ell}$-term in $\dd_{\Pi(\ell-1)}(x)$ is zero in \ref{partialdef}(\ref{p3})  ensures that  $\Pi(\ell-1) \not =\Pi'(\ell)$. This and \ref{partialdef}(\ref{p2}) for $\Pi$ yield \ref{partialdef}(\ref{p2}) for $\Pi'$. 
		Next, we define $a_{\ell}$ to be the smallest positive integer such that the $a_{\ell}b_{\ell}$ term in $\dd_{\gamma}(x)$ is $0$. Again, this guarantees \ref{partialdef}(\ref{p3}) for $\Pi'$.
We must now show that the equations (\ref{ematchesd}) known inductively for $(\aa,\Pi)$ imply those needed for $(\aa',\Pi')$.

To clarify what this means, let $[a,b): =\{ a, a+1, \ldots, b-1\}$.  For a die $\dd$, define $S(\dd, a, b)$ to be the set of degrees $s'\in [a,b) $ of nonzero terms in $\dd$. In these terms, (\ref{ematchesd}) for $(\aa,\Pi)$ determines $S(\dd, 0, b_\ell)$ for all $\dd$ and what we have to check is the following.

\begin{clm}\label{gammaclaim}~ 
	\begin{enumerate}
		\item\label{gamma1} For $r\in[1,a_\ell)$, $S(\dd_\gamma, rb_\ell, (r+1)b_\ell)= rb_\ell+ S(\dd_\gamma, 0, b_\ell)$.
		\item\label{gamma2} For $r\in[1,a_\ell)$ and all $g \not= \gamma$, $S(\dd_g, rb_\ell, (r+1)b_\ell)= \emptyset$.
	\end{enumerate}
\end{clm}

Figure~\ref{missingterms} spells out Claim~\ref{gammaclaim} visually. Each line describes a die with the first line giving $\gamma$. The thick black vertical segment indicates $b_\ell$ and the thinner one(s) its multiples. Inductively known terms to the left of the $b_\ell$-bar are indicated by dots starting in degree $0$ on the left, with small dots for zero terms and large dots for nonzero ones. The squares and triangles to the right of the $b_\ell$-bar are the terms in $\dd_\gamma$ that we need to show are present, generalizing the ``missing'' $x^7$ in the example at the start of this section. All the small dots to the right of the $b_\ell$-bar are terms we need to show are zero. 

\enlargethispage{-\baselineskip}
\enlargethispage{-\baselineskip}

\vskip6pt 

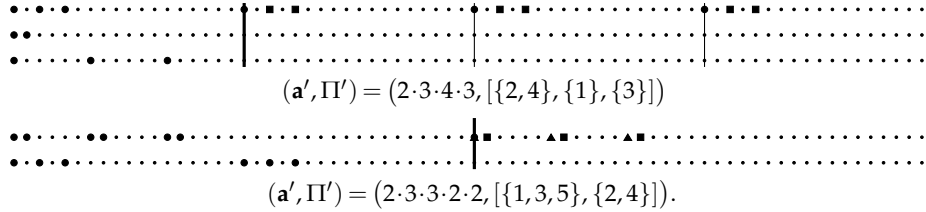
\numberedinsertionbot{missingterms}{Figure}{Known and missing terms as predicted by equation (\ref{ematchesd}).}{%
\begin{tikzpicture}[scale=0.17]{
	\tikzstyle{zerodot}=[fill=black];
	\tikzstyle{onedot}=[fill=black];
	\tikzstyle{jump}=[semithick, color=red, -stealth];
	\tikzstyle{jumps}=[semithick, color=blue, -stealth];
	\tikzstyle{prejump}=[semithick, color=red];
	\tikzstyle{prejumps}=[semithick, color=blue];
	\foreach \x in {0,...,71} 
		\filldraw[zerodot] (\x,0) circle (0.1);
	\foreach \x in {0,...,71} 
		\filldraw[zerodot] (\x,-2) circle (0.1);
	\foreach \x in {0,...,71} 
		\filldraw[zerodot] (\x,-4) circle (0.1);
	\foreach \v in {0}{
		\foreach \u in {0,...,2}{
			\filldraw[onedot] (2*\u+18*\v,0) circle (0.25);
			}
		}
	\foreach \v in {1,...,3}{
		\foreach \u in {0}{
			\filldraw[onedot] (2*\u+18*\v,0) circle (0.25);
			}
		\foreach \u in {1,...,2}{
			\filldraw[color=black] (2*\u+18*\v-0.25,-0.25) rectangle (2*\u+18*\v+0.25,0.25);
			}
		\draw[] (18*\v,0.5)--(18*\v,-4.5);
		}
	\draw[very thick] (18,0.5)--(18,-4.5);
	\foreach \u in {0,...,1} 
		\filldraw[onedot] (\u,-2) circle (0.25);
	\foreach \u in {0,...,2} 
		\filldraw[onedot] (6*\u,-4) circle (0.25);
	\draw[] (36,-6.5) node[]{\small$(\aa',\Pi'){\,=\,}\bigl(2{\kern0.5pt\cdot\kern0.5pt} 3{\kern0.5pt\cdot\kern0.5pt} 4{\kern0.5pt\cdot\kern0.5pt} 3, [\{2,4\},\{1\},\{3\}]\bigr)$};
	
	\foreach \x in {0,...,71} 
		\filldraw[zerodot] (\x,-10) circle (0.1);
	\foreach \x in {0,...,71} 
		\filldraw[zerodot] (\x,-12) circle (0.1);

	\foreach \w in {0}{
		\foreach \v in {0,...,2}{
			\foreach \u in {0,...,1}{
				\filldraw[onedot] (1*\u+6*\v+36*\w,-10) circle (0.25);
				}
			}
		}
	\foreach \w in {1}{
		\foreach \v in {0,...,2}{
			\foreach \u in {0}{
				\filldraw[color=black] (1*\u+6*\v+36*\w-0.30,-10-0.25) -- (1*\u+6*\v+36*\w,-10+0.25) -- (1*\u+6*\v+36*\w+0.30,-10-0.25) -- cycle;
				}
			\foreach \u in {1}{
				\filldraw[color=black] (1*\u+6*\v+36*\w-0.25,-10-0.25) rectangle (1*\u+6*\v+36*\w+0.25,-10+0.25);
				}
			}
		\foreach \v in {1,...,2}{
			}
		\filldraw[onedot] (36,-10) circle (0.25);
		\draw[very thick] (36*\w,-8.5)--(36*\w,-12.5);
		}

	\foreach \v in {0,...,1}{
		\foreach \u in {0,...,2}{
			\filldraw[onedot] (2*\u+18*\v,-12) circle (0.25);
			}
		}
	\draw[] (36,-14.5) node[]{\kern-2pt\small$(\aa',\Pi'){\,=\,}\bigl(2{\kern0.5pt\cdot\kern0.5pt} 3{\kern0.5pt\cdot\kern0.5pt} 3{\kern0.5pt\cdot\kern0.5pt} 2 {\kern0.5pt\cdot\kern0.5pt} 2, [\{1,3,5\},\{2,4\}]\bigr)$.};
	
	}
\end{tikzpicture}%
}

We will check Claim~\ref{gammaclaim} by a ``per-vertical-bar'' induction on $r$. The inductive step follows, by a second
``per-square-and-triangle'' induction on $s$, from the following refined claim.

\begin{clm}\label{gammaclaims} For $s \in S(\dd_\gamma, 0,b_\ell)$,
$S(\dd_\gamma, r b_\ell, r b_\ell+ s)= rb_\ell+ S(\dd_\gamma, 0, s)$.
\end{clm}

This claim for a given $s$ (that is, for one of squares or triangles in the top $\gamma$ row of Figure~\ref{missingterms}) 
shows that we obtain exactly the totals in the range $[rb_\ell, rb_\ell+s)$ from those in $[0,s)$ by replacing a $\dd_\gamma$ factor $x^{s'}$ by $x^{rb_\ell+s'}$ and leaving all factors from other dice unchanged. By Uniqueness of Terms, no die $\dd_g$ except $\dd_\gamma$ can have any term of degree in the range $[rb_\ell, rb_\ell+s)$. This gives Claim~\ref{gammaclaim}(\ref{gamma1}) up to degree $rb_\ell +s$ and shows that products of terms of smaller degree do \emph{not} yield the total $rb_\ell+s$. By Uniqueness of Terms, some die must contain an $x^{rb_\ell+s}$ term. To complete the induction, we need to check that this die must be~$\dd_\gamma$.  

Here is the key step. Pick the largest $h$ in $\pi_\gamma$ for which $rb_\ell+s$ is divisible by $b_h$. This $h$ is also the smallest $h$ for which the $s$-term of $\dd_\gamma$ picks up a term of positive degree from the factor $\Psi_{a_h}(x^{b_h})$ of $\dd_\gamma$. By construction, $s':= s-b_h$ is also in $S(\dd_\gamma, 0, s)$ and, by our induction on $s$, we know that $\dd_\gamma$ contains a term of degree $rb_\ell+s'$. If $\Pi(h+1)=g$, then, by interval freeness, $g \not=\gamma$. By construction, we therefore know that $\dd_{g}$ contains a term of degree $b_{h+1}$ and that $\dd_{\gamma}$ contains a term of degree $(a_{h}-1)b_{h}=(b_{h+1}-b_{h})$. 

Figure~\ref{commongaps} illustrates these choices for the two examples in Figure~\ref{missingterms}. For each term $s \in rb_\ell+ S(\dd_\gamma, 0,b_\ell)$, one arrow is drawn going left from $s$ to the $s'= s-b_h$ that we know inductively to be nonzero in $\dd_\gamma$  and a second is drawn from the $(b_{h+1}-b_{h})$-term of $\dd_\gamma$ going down and right to the $b_{h+1}$-term of $\dd_g$. Terms with the same value of $h$ use the same marker (triangle or square) and line style (straight or snaked).

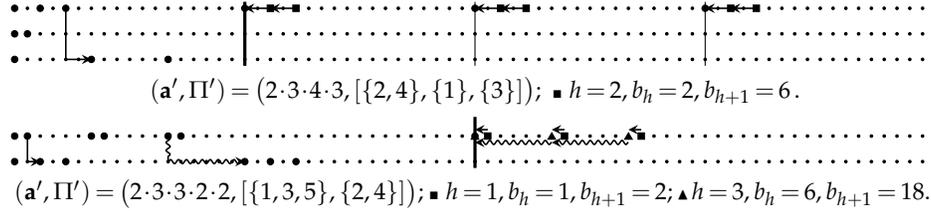
\numberedinsertionbot{commongaps}{Figure}{Nonzero coefficients of degrees differing by $+b_h$ and by $-b_h$.}{%
\begin{tikzpicture}[scale=0.17]{
	\tikzstyle{zerodot}=[fill=black];
	\tikzstyle{onedot}=[fill=black];
	\tikzstyle{jump}=[semithick, color=black, -stealth];
	\tikzstyle{jumps}=[semithick, color=black, -stealth, snake=snake,segment amplitude=.75,
         segment length=3,line after snake=1.5];
	\tikzstyle{prejump}=[semithick, color=black];
	\tikzstyle{prejumps}=[semithick, color=black, snake=snake,segment amplitude=.75,
         segment length=3,line after snake=1.5];
	\foreach \x in {0,...,71} 
		\filldraw[zerodot] (\x,0) circle (0.1);
	\foreach \x in {0,...,71} 
		\filldraw[zerodot] (\x,-2) circle (0.1);
	\foreach \x in {0,...,71} 
		\filldraw[zerodot] (\x,-4) circle (0.1);
	\foreach \v in {0}{
		\foreach \u in {0,...,2}{
			\filldraw[onedot] (2*\u+18*\v,0) circle (0.25);
			}
		}
	\foreach \v in {1,...,3}{
		\foreach \u in {0}{
			\filldraw[onedot] (2*\u+18*\v,0) circle (0.25);
			}
		\foreach \u in {1,...,2}{
			\filldraw[color=black] (2*\u+18*\v-0.25,-0.25) rectangle (2*\u+18*\v+0.25,0.25);
			\draw[jump] (2*\u+18*\v,0)--(2*\u-2+18*\v,0);
			}
		\draw[] (18*\v,0.5)--(18*\v,-4.5);
		}
	\draw[very thick] (18,0.5)--(18,-4.5);
	\foreach \u in {0,...,1} 
		\filldraw[onedot] (\u,-2) circle (0.25);
	\foreach \u in {0,...,2} 
		\filldraw[onedot] (6*\u,-4) circle (0.25);
	\draw[prejump] (4,0)--(4,-4);\draw[jump] (4,-4)--(6,-4);
	\draw[] (36,-6.5) node[]{\small$(\aa',\Pi'){\,=\,}\bigl(2{\kern0.5pt\cdot\kern0.5pt} 3{\kern0.5pt\cdot\kern0.5pt} 4{\kern0.5pt\cdot\kern0.5pt} 3, [\{2,4\},\{1\},\{3\}]\bigr)${;~~{\fontsize{4}{4}\selectfont$\blacksquare\,$} $h{\,=\,}2, b_h{\,=\,}2, b_{h+1}{\,=\,}6\,.$}};
	
	\foreach \x in {0,...,71} 
		\filldraw[zerodot] (\x,-10) circle (0.1);
	\foreach \x in {0,...,71} 
		\filldraw[zerodot] (\x,-12) circle (0.1);

	\foreach \w in {0}{
		\foreach \v in {0,...,2}{
			\foreach \u in {0,...,1}{
				\filldraw[onedot] (1*\u+6*\v+36*\w,-10) circle (0.25);
				}
			}
		}
	\foreach \w in {1}{
		\foreach \v in {0,...,2}{
			\foreach \u in {0}{
				\filldraw[color=black] (1*\u+6*\v+36*\w-0.30,-10-0.25) -- (1*\u+6*\v+36*\w,-10+0.25) -- (1*\u+6*\v+36*\w+0.30,-10-0.25) -- cycle;
				}
			\foreach \u in {1}{
				\filldraw[color=black] (1*\u+6*\v+36*\w-0.25,-10-0.25) rectangle (1*\u+6*\v+36*\w+0.25,-10+0.25);
				}
			\draw[jump] (1*1+6*\v+36*\w,-9.5)--(1*1+6*\v+36*\w-1,-9.5);
			}
		\foreach \v in {1,...,2}{
			\draw[jumps] (6*\v+36*\w,-10.5)--(6*\v+36*\w-6,-10.5);
			}
		\filldraw[onedot] (36,-10) circle (0.25);
		\draw[very thick] (36*\w,-8.5)--(36*\w,-12.5);
		}

	\foreach \v in {0,...,1}{
		\foreach \u in {0,...,2}{
			\filldraw[onedot] (2*\u+18*\v,-12) circle (0.25);
			}
		}
	\draw[prejump] (1,-10)--(1,-12);\draw[jump] (1,-12)--(2,-12);
	\draw[prejumps] (12,-10)--(12,-12);\draw[jumps] (12,-12)--(18,-12);
	\draw[] (36,-14.5) node[]{\kern-2pt\small$(\aa',\Pi'){\,=\,}\bigl(2{\kern0.5pt\cdot\kern0.5pt} 3{\kern0.5pt\cdot\kern0.5pt} 3{\kern0.5pt\cdot\kern0.5pt} 2 {\kern0.5pt\cdot\kern0.5pt} 2, [\{1,3,5\},\{2,4\}]\bigr)$; {{\fontsize{4}{4}\selectfont$\blacksquare\,$} $h{\,=\,}1, b_h{\,=\,}1, b_{h+1}{\,=\,}2$};{{\tiny~$\blacktriangle\,$}$h{\,=\,}3, b_h{\,=\,}6, b_{h+1}{\,=\,}18.$}};

	}
\end{tikzpicture}%
}

We now obtain a contradiction to Uniqueness of Terms if there is a term of degree $rb_\ell+s$ in any $\dd_{g'}$ with $g'$ not equal to $ \gamma$ (but possibly equal to $g$). Indeed, we would be able to produce terms of degree 
$rb_\ell+s+b_{h+1}-b_{h}$ in two distinct ways: first, using the terms of degrees $b_{h+1}-b_{h}$ in $\dd_\gamma$ and $rb_\ell+s$ in $\dd_{g'}$ and the constant terms from all other dice; and, second, using the terms of degrees $rb_\ell+ s-b_h$ in $\dd_\gamma$ and $b_{h+1}$ in $\dd_{g}$ and the constant terms from all other dice. Therefore, $\dd_\gamma$  must contain an $x^{rb_\ell+s}$ term as claimed in (\ref{gammaclaim}) and Theorem~\ref{mainthm} follows. 

Figure~\ref{doubleterm} illustrates this last step graphically, following Figures~\ref{missingterms} and~\ref{commongaps}. A~potential term of degree ${rb_\ell+s}$ in a $\dd_{g'}$ with $g'\not = \gamma$ is indicated by large black circle. A styled path joins this circle to the known term in $\dd_\gamma$ of degree smaller by $b_h$. A second path in the same style goes down and across between known nonzero terms of degrees less than $b_\ell$ and differing by $b_h$ as in Figure~\ref{commongaps}. Several potential terms may share a ``down-and-across'' path, as happens in the top example. For each potential term, the common total of the smaller degree from either of its paths and the larger degree from the other is marked by a vertical segment in the common style. 

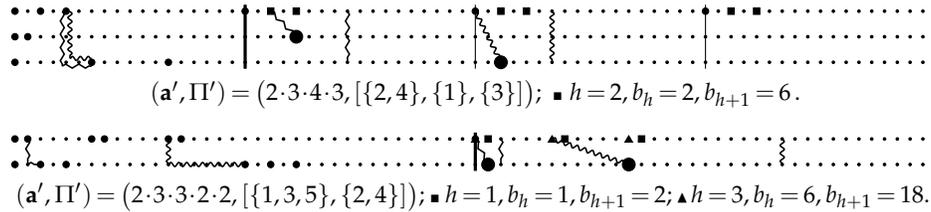
\numberedinsertionbot{doubleterm}{Figure}{Duplicated totals when an $x^s$ term occurs in a die $\dd_g$ other than $\dd_\gamma$.}{%
	\begin{center}\begin{tikzpicture}[scale=0.17]{
		\tikzstyle{zerodot}=[fill=black];
		\tikzstyle{onedot}=[fill=black];
		\tikzstyle{ordot}=[thick, fill=orange, color=orange];
		\tikzstyle{grdot}=[thick, fill=green, color=green];
		\tikzstyle{orrdot}=[very thick, fill=orange, color=orange];
		\tikzstyle{grrdot}=[very thick, fill=green, color=green];
	\tikzstyle{jump}=[semithick, color=black, snake=zigzag,segment amplitude=.75,
         segment length=5,line after snake=0];
	\tikzstyle{jumps}=[semithick, color=black, snake=snake,segment amplitude=.75,
         segment length=3,line after snake=0];
	\tikzstyle{prejump}=[semithick, color=black];
	\tikzstyle{prejumps}=[semithick, color=black, snake=snake,segment amplitude=.75,
         segment length=3,line after snake=1.5];
		\foreach \x in {0,...,71} 
			\filldraw[zerodot] (\x,0) circle (0.1);
		\foreach \x in {0,...,71} 
			\filldraw[zerodot] (\x,-2) circle (0.1);
		\foreach \x in {0,...,71} 
			\filldraw[zerodot] (\x,-4) circle (0.1);
		\foreach \v in {0}{
			\foreach \u in {0,...,2}{
				\filldraw[onedot] (2*\u+18*\v,0) circle (0.25);
				}
			}
		\foreach \v in {1,...,3}{
			\foreach \u in {0}{
				\filldraw[onedot] (2*\u+18*\v,0) circle (0.25);
				}
			\foreach \u in {1,...,2}{
			\filldraw[color=black] (2*\u+18*\v-0.25,-0.25) rectangle (2*\u+18*\v+0.25,0.25);
				}
			\draw[] (18*\v,0.5)--(18*\v,-4.5);
			}
			\draw[very thick] (18,0.5)--(18,-4.5);
		\foreach \u in {0,...,1} 
			\filldraw[onedot] (\u,-2) circle (0.25);
		\foreach \u in {0,...,2} 
			\filldraw[onedot] (6*\u,-4) circle (0.25);
	\draw[] (36,-6.5) node[]{\small$(\aa',\Pi'){\,=\,}\bigl(2{\kern0.5pt\cdot\kern0.5pt} 3{\kern0.5pt\cdot\kern0.5pt} 4{\kern0.5pt\cdot\kern0.5pt} 3, [\{2,4\},\{1\},\{3\}]\bigr)${;~~{\fontsize{4}{4}\selectfont$\blacksquare\,$} $h{\,=\,}2, b_h{\,=\,}2, b_{h+1}{\,=\,}6\,.$}};

		\foreach \x in {0,...,71} 
			\filldraw[zerodot] (\x,-10) circle (0.1);
		\foreach \x in {0,...,71} 
			\filldraw[zerodot] (\x,-12) circle (0.1);

		\foreach \w in {0}{
			\foreach \v in {0,...,2}{
				\foreach \u in {0,...,1}{
					\filldraw[onedot] (1*\u+6*\v+36*\w,-10) circle (0.25);
					}
				}
			}
		\foreach \w in {1}{
			\foreach \v in {0,...,2}{
				\foreach \u in {0}{
				\filldraw[color=black] (1*\u+6*\v+36*\w-0.30,-10-0.25) -- (1*\u+6*\v+36*\w,-10+0.25) -- (1*\u+6*\v+36*\w+0.30,-10-0.25) -- cycle;
					}
				\foreach \u in {1}{
				\filldraw[color=black] (1*\u+6*\v+36*\w-0.25,-10-0.25) rectangle (1*\u+6*\v+36*\w+0.25,-10+0.25);
					}
				}
			\foreach \v in {1,...,2}{
				}
			\filldraw[onedot] (36,-10) circle (0.25);
			\draw[very thick] (36,-9.5)--(36,-12.5);
			}

		\foreach \v in {0,...,1}{
			\foreach \u in {0,...,2}{
				\filldraw[onedot] (2*\u+18*\v,-12) circle (0.25);
				}
			}
			
		\filldraw[onedot] (22,-2) circle (0.5);
		\filldraw[onedot] (38,-4) circle (0.5);
		\draw[jump] (3.7,0) -- (3.7,-4.3);
		\draw[jump] (3.7,-4.3) -- (6,-4.3);
		\draw[jumps] (4.3,0) -- (4.3,-3.7);
		\draw[jumps] (4.3,-3.7) -- (6,-3.7);
		\draw[jump] (20,0) -- (22,-2);
		\draw[jump] (26,0) -- (26,-4);
		\draw[jumps] (36,0) -- (38,-4);
		\draw[jumps] (42,0) -- (42,-4);
		\filldraw[onedot] (37,-12) circle (0.5);
		\filldraw[onedot] (48,-12) circle (0.5);
		\draw[jump] (1,-10) -- (1,-12);
		\draw[jump] (1,-12) -- (2,-12);
		\draw[jump] (36,-10) -- (37,-12);
		\draw[jump] (38,-10) -- (38,-12);
		\draw[jumps] (12,-10) -- (12,-12);
		\draw[jumps] (12,-12) -- (18,-12);
		\draw[jumps] (42,-10) -- (48,-12);
		\draw[jumps] (60,-10) -- (60,-12);

		\draw[] (36,-14.5) node[]{\kern-2pt\small$(\aa',\Pi'){\,=\,}\bigl(2{\kern0.5pt\cdot\kern0.5pt} 3{\kern0.5pt\cdot\kern0.5pt} 3{\kern0.5pt\cdot\kern0.5pt} 2 {\kern0.5pt\cdot\kern0.5pt} 2, [\{1,3,5\},\{2,4\}]\bigr)$; {{\fontsize{4}{4}\selectfont$\blacksquare$} $h{\,=\,}1, b_h{\,=\,}1, b_{h+1}{\,=\,}2$};{{\tiny~$\blacktriangle\,$}$h{\,=\,}3, b_h{\,=\,}6, b_{h+1}{\,=\,}18.$}};
			
	}\end{tikzpicture}%
\end{center}
}
\vfill\eject 
\section{Applications and further questions.}\label{applications}

\subsection*{What dice lie in fair sacks?} Viewing the Gasarch--Kruskal Theorem as saying that ``Every die in a fair sack is semifair'' naturally suggests the question, ``Does every semifair die occur in a fair sack?''. The answer is usually negative, and the simplest examples are the dice $\ee_s(x) = 1+x^s+x^{2s-1}+x^{3s-1}$ for $s\ge2$. 

For $\ee_2(x)= 1+x^2+x^3+x^5$, we simply have to note that, by Uniqueness of Terms, the sack must also contain a die of the form $(1+ x+\cdots)$ and then the total $x^3$ arises in two ways. For $\ee_3(x)= 1+x^3+x^5+x^8$, we need to argue that, since there is a unique die with an $x$ term, an $x^2$ term cannot arise as a product of lower degree factors. Hence there is a also a die with an $x^2$ term and this would give two ways to obtain $x^5$. Similar arguments fail, however, for $\ee_4(x) = 1+x^4+x^7+x^{11}$ because there are fair sacks for which no die has an $x^3$ term. However, if so, then $x^3$ must arise as a product of lower degree factors and hence the $x$ and $x^2$ terms occur in two \emph{different} dice in the sack, again giving us two ways to produce $x^7$. 
As $s$ increases, ruling out $\ee_s(x)$ requires considering increasingly large numbers of other dice.  Theorem~\ref{mainthm} provides a criterion that lets us read off, directly from $\dd(x)$, whether it lies in a fair sack and that immediately shows that no $\ee_s(x)$ does.

\begin{cor}\label{degreedivisibility} A die $\dd$ lies in a fair sack if and only if $\dd(x) = \prod_{j=1}^{m} \Psi_{a_j}(x^{b_j})$ with $a_j b_j \divides b_{j+1}$ for $1\le j < m$. In particular, the degrees of all nonzero terms of $\dd(x)$ are multiples of the smallest positive such degree.
\end{cor}
\begin{proof} Any die in a partition-factorization sack has the claimed form, hence the first statement follows from Theorem~\ref{mainthm}. It immediately implies the second.
\end{proof}

As an example of a new restriction on the orders of dice in a fair sack, we sharpen Corollary~9 of~\cite{GasarchKruskal} which shows that a fair sack with total $t$ must contain a die of order at least $\phi(t)+1$ where $\phi$ is the Euler totient function.

\begin{cor}\label{largeorder} If $p$ is the smallest prime dividing $t$, then a fair sack with total $t$ always contains a die of order at least $t(1-\frac{1}{p})+1$. In particular, every fair sack with total $t$ contains a die of order at least $\frac{t}{2}+1$.	
\end{cor}
\begin{proof} Realize $\SS$ as a partition-factorization sack arising from a prime factorization $\aa$ of length $\ell$ using Proposition~\ref{mainconstruction}.(\ref{fpprime}). The polynomial $\Psi_{a_\ell}(x^{b_\ell})$ is a factor of the $\dd_{g}(x)$ associated to the part $\pi_g$ containing $\ell$. Since the degree of $\Psi_{a_\ell}(x^{b_\ell})$ is equal to $(a_\ell-1)b_\ell = t-{b_\ell}= t(1-\frac{1}{a_\ell})$, the degree of $\dd_{g}(x)$ is at least this large. Since $a_\ell \ge p$, the corresponding die has order at least $t(1-\frac{1}{p})+1$.
\end{proof}

\subsection*{Algorithmic aspects.} Theorem~\ref{mainthm} does a bit more than show that any fair sack arises from the constructions of \sect{details}. Its proof amounts to an algorithm for finding the factorization and interval free partition from which it arises. 

Likewise, Corollary~\ref{degreedivisibility} yields an algorithm for determining whether a given sack $\SS$ of semifair dice is fair that is more efficient than the brute force check of the uniqueness of all totals suggested in~\cite[Corollary~6]{GasarchKruskal}. Since such algorithms are of purely theoretical interest, we only sketch the idea, leaving details to the reader. 

The first step is to check that each die $\dd$ has the form of the Corollary, by a procedure like that in the proof of Theorem~\ref{mainthm}. For example, if $b_1$ is the lowest degree of a term occurring in $\dd$ and $a_1$ is the smallest positive number for which $\dd$ has not have a $a_1 b_1$ term, then  $\Psi_{a_1}(x^{b_1})$ must divide $\dd(x)$. If it does, we repeat this test for the quotient, inductively producing the sequence of $a_j$ and $b_j$ of the Corollary and stopping when the quotient is $1$. 

If each die in $\SS$ passes these tests, we let $P_\SS$ be the disjoint union of the sets of pairs $(a_j,b_j)$ for all dice $\dd$ in $\SS$ ordered so that the $b_j$ are nondecreasing, and set $\ell := |P_\SS|$ and $b_{\ell+1} := a_\ell b_\ell$. Then $\SS$ is fair exactly when this yields an ordered factorization of its total $t$: that is, $b_1=1$, $b_{h+1} = a_{h}b_{h}$ for  $1 \le h \le \ell$ and $b_{\ell+1}=t$.


\subsection*{Atomizations.} Finally, each die in a partition-factorization sack is itself the total die of the subsack determined by its part. This motivates the following definition which leads to our most striking corollary. 

A die is \emph{atomic} if it is not the total die of any sack of size $2$ or more---equivalently, if $\dd(x)$ does not factor in $\RR^+[x]$. A sack is atomic if all its dice are. Every die $\dd$ is the total die of an atomic sack, that we call an \emph{atomization} of $\dd$, by a standard argument. (If $\dd$ is not itself atomic, then it is the total die of a sack of dice, all of strictly smaller orders. By induction, these have atomizations whose union is an atomic sack with total~$\dd$.) 

Note, however, that atomizations are  usually \emph{not} unique. For example, in view of (\ref{atomicsfsack}) of Corollary~\ref{atomic}, the die in the $12$ line of Table~\ref{twelvefair} has $3$ atomizations, given in the $2\cdot2\cdot3$, $2\cdot3\cdot2$, and $3\cdot2\cdot2$ lines. The atomizations of a sack $\SS$ are the sacks obtained by atomizing, in any way, all the dice in $\SS$.

\begin{cor}\label{atomic}$\,$ \begin{enumerate}
	\item Any atomization of a fair sack is fair. 
	\item \label{atomicsfdie} The atomic dice that lie in some fair sack are those of the form $\Psi_p(x^b)$ with $p$ prime. 
	\item \label{atomicsfsack} The atomic fair sacks are the factorization sacks of ordered prime factorizations~$\aa$.
	\item Every atomic fair sack of size $m$ contains a unique fair subsack $\SS_{m'}$ of each size $m' \le m$ consisting of dice associated to the first $m'$ factors in $\aa$.
\end{enumerate}
\end{cor}
\begin{proof} The first claim holds because totals are preserved under atomization. The  proof of Proposition~\ref{mainconstruction} shows that only sacks associated to prime factorizations can be atomic. If a die in such a sack was not itself atomic, then by atomizing it we would obtain a fair sack contradicting Theorem~\ref{mainthm}. This proves the second and third assertions. Lemma~\ref{partialfairness} implies the fairness of the subsacks $\SS_{m'}$ in the last statement. Uniqueness follows by induction on $m$. If a fair subsack $\SS'$ of size $m'<m$ does not contain the die $\dd_m$ associated to the last factor in $\aa$, its intersection with $\SS_{m-1}$ is fair subsack of size $m'$ that, inductively, must equal $\SS_{m'}$. If $\dd_m \in \SS'$, we get a contradiction by removing it to produce a fair subsack $\SS''$ of $\SS_{m-1}$ of size $m'-1$ . By induction, $\SS''= \SS_{m'-1}$ which does not contain $\dd_{m'-1}$. Now, adding $\dd_m$ back to $\SS_{m'-1}$ to get $\SS'$ yields an unfair sack because the total $b_{m'-1}$ does not occur.
\end{proof} 

\enlargethispage{-\baselineskip}
\enlargethispage{-\baselineskip}

\subsection*{Closing questions for the reader.} We conclude by posing a few questions to the reader. We can factor the dice $\ee_s(x)$ defined above as $\ee_s(x)=(1+x^{s})(1+x^{2s-1})=\Psi_2(x^s)\Psi_2(x^{2s-1})$. This is an atomization by Corollary~\ref{atomic}.(\ref{atomicsfdie}) and, at least for small $s$, there are no others.\footnote{The roots of $\Psi_2(x^k)=1+x^k$ are those $2k${th} roots of unity that are not $k${th} roots of unity, from which its irreducible factors over $\RR$ are easily found. To find all atomizations of $\dd_s(x)$ by brute force, we simply enumerate all partitions of the factors for both $k=s$ and $k=s-1$, find those for which the product of the factors in each part has all coefficients nonnegative (and hence gives a die polynomial), and eliminate any that themselves contain nonatomic dice.} All other atomizations of semifair dice not lying in a fair sack that we have found contain only semifair dice, but this does not follow from (\ref{atomicsfdie}) above and we have not found any proof. So we ask the reader, ``\emph{Must every semifair die have a semifair atomization?}'' or, more greedily, ``\emph{Is semifairness closed under atomization?}''

In a related direction, we may define, following Chapman and McClain~\cite{ChapmanMcClain}, the elasticity of a polynomial in $\RR^+[x]$ to be the maximum of the ratios $\frac{n}{n'}$ for which the polynomial has an atomization with $n$ atoms and a second with $n'$. Examples are given in \cite{ChapmanMcClain} of polynomials having elasticity equal to any rational number $r\ge 1$. But while $\Psi_t(x)$ has many atomizations by Corollary~\ref{atomic}, the number of atoms in all of them is the number of prime factors of $t$. That is, $\Psi_t(x)$ has elasticity $1$. So we close by asking, ``\emph{What are the possible elasticities of more general semifair polynomials?}'' or, more greedily again, ``\emph{Do all semifair polynomials have elasticity exactly $1$?}''

\subsection*{Acknowledgements.} Part of this work was completed while the author was on a Fordham University Faculty Fellowship and 
visiting the Vietnam Institute for Advanced Studies in Mathematics, Galatasaray University and the Universitá di Firenze. I thank all of these institutions for their support, and \begin{otherlanguage*}{vietnamese}Nguyễn Hữu Dự\end{otherlanguage*}, Ayberk Zeytin and Giorgio Ottaviani for their personal hospitality. I am also grateful to Abe Smith and Dave Swinarski for helpful conversations and to the two anonymous referees for their careful reading of an earlier draft and for suggesting a number of improvements.

\bibspread

\parindent0pt\par
\end{document}